\documentclass[a4paper,12pt]{amsart}
\usepackage[mathscr]{eucal}
\usepackage{amssymb}
\usepackage{latexsym}

\theoremstyle{plain}
\newtheorem{thm}{Theorem}[section]
\newtheorem{prop}[thm]{Proposition}
\newtheorem{lem}[thm]{Lemma}

\theoremstyle{definition}
\newtheorem{defn}[thm]{Definition}
\newtheorem{rem}[thm]{Remark}

\evensidemargin\paperwidth
\advance\evensidemargin-\textwidth
\advance\oddsidemargin-1in
\evensidemargin\oddsidemargin

\topmargin\paperheight
\advance\topmargin-\textheight
\advance\topmargin-1in

\begin{document}

\title[Quasi-homomorphism and scl]{On quasi-homomorphisms and commutators in the special linear group over a euclidean ring}
\author{Masato MIMURA}
\date{}

\maketitle

\begin{abstract}
We prove that for any euclidean ring $R$ and $n\geq 6$, $\Gamma = SL_n (R)$ has no unbounded quasi-homomorphisms. By Bavard's duality theorem,
this means 
that the stable commutator length vanishes on $\Gamma$. The result is
particularly
interesting for $R = F[x]$ for a certain field $F$ (such as $\mathbb{C}$), because in this case the commutator length
on $\Gamma$ is known to be unbounded. This answers a question of M. Ab\'{e}rt
and N. Monod for $n\geq6$.
\end{abstract}

\section{\textbf{Introduction and main results}}\label{sec:Intro}
Throughout this paper, let $G$, $H$, and $\Gamma$ be discrete groups (in some case uncountable) and we always assume that fields are commutative. For a group $\Gamma$, let $[\Gamma , \Gamma ]$ denote the commutator subgroup of $\Gamma$. For an associative ring $B$ with unit and $n \in \mathbb{N}$, let $I_n$ be the unit of the matrix ring $M_n (B)$ and $1=I_1$. We shall basically use the symbol $B$ for an associative (and not necessarily commutative) ring with unit, the symbol $A$ for an  associative and commutative ring with unit, and the symbol $R$ for a euclidean ring. For an associative and commutative ring $A$ with unit, $SL_n(A)$ denotes the multiplicative group of all elements in $M_n(A)$ whose determinants are $1$. 

A \textit{quasi-homomorphism} on $\Gamma$ is a function $\psi \colon \Gamma \to \mathbb{R}$ such that 
$$
\Delta (\psi) := \sup_{g,h \in \Gamma} |\psi (gh)- \psi (g) -\psi (h) | < \infty, 
$$
and $\Delta (\psi)$ is called the \textit{defect} of $\psi$. 
A quasi-homomorphism $\psi$ on $\Gamma$ is said to be \textit{perturbed from a homomorphism} (or simply, \textit{perturbed}) if there exists a decomposition $\psi =\psi_0 + c$ such that $\psi_0 \colon \Gamma \to \mathbb{R}$ is a homomorphism and $c \colon \Gamma \to \mathbb{R}$ has bounded image. We define the following vector space: 
\begin{equation*}
\widetilde{QH}(\Gamma ):= \{\mathrm{all\ quasi}\textrm{-}\mathrm{homomorphisms\ on\ }\Gamma \} / 
\{\mathrm{all\ perturbed\ quasi}\textrm{-}\mathrm{homomorphisms\ on\ }\Gamma \}.
\end{equation*}
D. B. A. Epstein and K. Fujiwara \cite{EF} have proved that any non-elementary hyperbolic group has infinite dimensional $\widetilde{QH}$. In contrary, M. Burger and N. Monod \cite{BM} have shown that for higher rank lattices $\Gamma$, such as $SL_n (\mathbb{Z})$ $(n \geq 3)$, $\widetilde{QH}(\Gamma )$=0.

Investigations of $\widetilde{QH}$ are of importance because $\widetilde{QH}$ is closely related to the concept of the \textit{stable commutator length}, abbreviated as \textit{scl}, of a group. First we recall the definition of the stable commutator length. Basic references on scl are \cite{Bav} and \cite{Cal}. 

\begin{defn}\label{def:scl}
Let $\Gamma$ be a group. Let $g$ be any element in $[\Gamma ,\Gamma]$.
\begin{itemize}
  \item  The \textit{commutator length} of $g$ in $[\Gamma ,\Gamma]$ 
  (which is written as $\mathrm{cl}(g)$) is the least number of commutators whose product is equal to $g$. Here for unit $e_\Gamma$ in $\Gamma$, we define $\mathrm{cl}(e_\Gamma)=0$. 
  \item  The \textit{scl} (\textit{stable commutator length}) of $g$ in $[\Gamma ,\Gamma]$ is defined by 
  $$
  \mathrm{scl}(g)=\lim_{n\to \infty}\frac{\mathrm{cl}(g^n)}{n}.
  $$
\end{itemize}
\end{defn}
The following relation between $\widetilde{QH}$ and scl is given by Bavard's duality theorem \cite[p.111]{Bav}.  
\begin{thm}\label{thm:scl}$($\cite{Bav}$)$
For a group $\Gamma$, the following are equivalent: 
\begin{itemize}
 \item[$(i)$] The equality $\widetilde{QH}(\Gamma)=0$ holds.
 \item[$(ii)$] The scl $($on $[\Gamma ,\Gamma]$$)$ vanishes identically. 
\end{itemize}
\end{thm}

The goal of this paper is to provide a new example of $\Gamma$ that satisfies the conditions in Theorem~\ref{thm:scl}. Before the presentation of our main result, we shall recall basic definitions of elementary linear groups and bounded generation. We recall that for an associative ring $B$ with unit, the \textit{elementary linear group} $EL_n (B)$ is defined as the multiplicative group generated by all elementary matrices in $M_n (B)$. Here an \textit{elementary matrix} in $M_n (B)$ is an  element whose diagonal entries are $1$ and other entries except one are all $0$. (It is also common to call $EL_n(B)$ the \textit{elementary group} over $B$ and use the symbol $E_n(B)$ for this group.) Note that $[EL_n (B), EL_n (B)]=$$EL_n(B)$ for $n\geq 3$. We say that subsets $F_1$, \ldots , $F_j$ of a group $\Gamma$ \textit{boundedly generate} $\Gamma$ if there exists a constant $M \in \mathbb{N}$ such that 
$$
\Gamma = F_{i(1)}F_{i(2)} \cdots F_{i(M)}
$$
holds, where each $i(m)$ $(1\leq m \leq M)$ is in $\{ 1, \ldots ,j \}$. 
Now we present our main result. 

\begin{thm}\label{thm:ge2}$($Main Theorem\,$)$
Let $A$ be an associative and commutative ring with unit. Let $n\geq 6$ be an integer and $\Gamma = EL_n (A)$. Suppose in addition the following two conditions hold:
\begin{itemize}
   \item[$(1)$] The elementary linear group $EL_2 (A)$ coincides with $SL_2 (A)$. 
   \item[$(2)$] The group $\Gamma$ is boundedly generated by $G=SL_2 (A)$ and the set of all single commutators in
    $\Gamma$. Here we regard $G$ as the subgroup of 
   $\Gamma$ in the left upper corner. 
\end{itemize}
Then all quasi-homomorphisms on $\Gamma$ are bounded. Or equivalently, the scl vanishes on $\Gamma$. 
\end{thm}
We mention that P. M. Cohn \cite{Coh} has defined the concept of the $GE_2$\textit{-ring} for  an associative (and not necessarily commutative) ring $B$ with unit as follows: the ring $B$ is called a $GE_2$-ring if $GE_2 (B)$ coincides with $GL_2 (B)$. Here $GL_2 (B)$ denotes the group of all invertible elements in $M_2 (B)$, and $GE_2 (B)$ denotes the subgroup of $GL_2 (B)$ generated by its elementary matrices and invertible diagonal matrices. If $B$ is commutative in addition, this definition is equivalent to condition $(1)$ in Theorem~\ref{thm:ge2}. For the proof of this equivalence, one uses the standard form on $GE_2(B)$, see \cite[p.10]{Coh}.

A well-known example of rings satisfying the conditions of Theorem~\ref{thm:ge2} (for all $n$) is a euclidean ring (we will see this in Section~\ref{sec:Thm}). Thus, we obtain the following theorem. We note that for any euclidean ring, $EL_n=SL_n$ for all $n\geq 2$. 

\begin{thm}\label{thm:euclid}
Let $R$ be a euclidean ring and $\Gamma = SL_n (R)$. Then for any $n\geq 6$, the conclusions of Theorem~\ref{thm:ge2} hold. 
\end{thm}

Theorem~\ref{thm:ge2} follows from the following Theorem~\ref{thm:vanish}, which itself is of interest. 

\begin{thm}\label{thm:vanish}
Let $A$ be an associative and commutative ring with unit and $\Gamma= EL_6(A)$. Let $G=EL_2 (A)$ be the subgroup of $\Gamma$ in the left upper corner. If $\phi$ be a $\mathrm{homogeneous}$ $\mathrm{quasi}$-$\mathrm{homomorphism}$ $($see Definition~\ref{def:homoge}$)$ on $\Gamma$, then the restriction of $\phi$ on $G$ vanishes identically. 
\end{thm}

A particularly interesting case of Theorem~\ref{thm:euclid} is where $R=F[x]$ for a field $F$ of infinite transcendence degree over its prime field. (For instance, $F=$$\mathbb{R}$ or $\mathbb{C}$.)  For $n \geq 6$ and $F$ as above, the group $SL_n (F[x])$ satisfies the following property: although the commutator length is unbounded, the stable commutator length vanishes identically. We refer to \cite{DV} and \cite{Kal} for the proof of the unboundedness of the commutator length in this case. We note that A. Muranov \cite{Mur} has shown that there is a simple and finitely generated group with the property above, and that P.-E. Caprace and K. Fujiwara \cite{CF} have investigated the non-vanishing of $\widetilde{QH}$ (and hence also non-vanishing of scl) for Kac--Moody groups. Theorem~\ref{thm:euclid} particularly answers, for $n \geq 6$, the question of M. Ab\'{e}rt and N. Monod \cite[Problem Q]{Mon2} whether $SL_n (F[x])$ for the field $F$ as in above permits a non-perturbed quasi-homomorphism.

We end this introduction by noting the following remark: M. Gromov observed in \cite{Gr} that for a discrete group $\Gamma$, the vanishing of $\widetilde{QH}$ is also equivalent to the injectivity of the natural map 
$H^2_b (\Gamma ; \mathbb{R}) \to H^2 (\Gamma ; \mathbb{R})$, 
where the left-hand side is the second \textit{bounded cohomology}. 
For details of bounded cohomology, we refer to \cite{Gr}, \cite{Mon1} and \cite{Mon2}.

\section{\textbf{Preliminaries on homogeneous quasi-homomorphisms}}\label{sec:pre}
\begin{defn}\label{def:homoge}
A quasi-homomorphism $\phi$ on $\Gamma$ is said to be \textit{homogeneous} if for any $g\in \Gamma$ and $m\in \mathbb{Z}$, $\phi (g^m) =m \cdot \phi (g)$ holds. 
\end{defn}

We need the following two basic facts on quasi-homomorphisms. See \cite[Subsection 2.2]{Cal} for a comprehensive treatment. 

\begin{prop}\label{prop:hom}
Let $\psi$ be a quasi-homomorphism on $\Gamma$. Then there exists a $\mathrm{homogeneous}$ quasi-homomorphism $\overline{\psi}$ such that the following holds: $\psi$ is perturbed if and only if $\overline{\psi}$ is a homomorphism. 
\end{prop}

\begin{lem}\label{lem:hom}
Let $\phi$ be a homogeneous quasi-homomorphism on $\Gamma$. Then the following hold: 
\begin{itemize}
   \item[$(1)$]  For any commuting pair $g,h \in \Gamma$, $\phi (gh)= \phi (g)+ \phi (h)$.
   \item[$(2)$]  The map $\phi$ is constant on each conjugacy class, namely, 
   for any $g\in \Gamma$ and any $t \in \Gamma$, $\phi (tgt^{-1})=\phi(g)$ holds.
\end{itemize}
\end{lem}

\begin{proof}
We only present a proof of $(2)$. Let $\Delta $ be the defect of $\phi$. By noticing $(tg t^{-1})^{m}=tg^m t^{-1}$,  one observes that for any $m \in \mathbb{N}$, $m \cdot | \phi (tgt^{-1})-\phi(g) |$$\leq 2\Delta $ holds. Hence, $\phi (tgt^{-1})=\phi(g)$.
\end{proof}

The following lemma is an immediate result from Lemma~\ref{lem:hom}. 
\begin{lem}\label{lem:vanish}
Let $\phi$ be a homogeneous quasi-homomorphism on $\Gamma$. If $g\in \Gamma$ is conjugate to its inverse, then $\phi (g)=0$.
\end{lem}

We proceed to preliminaries on homogeneous quasi-homomorphism on elementary linear groups. Let $B$ be an associative ring with unit and $n\geq 2$. We call an element $g \in EL_n (B)$ a \textit{unit upper }(respectively \textit{lower}) \textit{triangular matrix} if all diagonal entries are $1$ and all of the entries below (respectively above) the diagonals are $0$. We define $U_n B$ (respectively $L_n B$) as the group of all unit upper (respectively lower) triangular matrices of degree $n$. 

\begin{lem}\label{lem:triang}
Let $\phi$ be a homogeneous quasi-homomorphism on $\Gamma =EL_n (B)$. If $n \geq 3$, then the following hold: 
\begin{itemize}
 \item[$(1)$]  For any elementary matrix $s\in \Gamma$, $\phi (s)=0$.
 \item[$(2)$]  For any $h \in (U_n B) \cup (L_n B)$, 
 $\phi (h)=0$.
\end{itemize}
\end{lem}

\begin{proof}
$(1)$ follows from Lemma~\ref{lem:vanish} because for $n\geq 3$, an elementary matrix is conjugate to its inverse. Hence, $\phi$ is bounded on $U_n B$ and $L_n B$. From the homogeneity of $\phi$, one obtains $(2)$.
\end{proof}
 
The following observation plays an important role in this paper. 
\begin{lem}\label{lem:gene}
Let $\Gamma$ be a group and $H< \Gamma $ be  a subgroup. Let  $\phi$ be a homogeneous quasi-homomorphism on $\Gamma$. Suppose  $\phi$ vanishes on $H$. Then for  any $h\in H$ and any $g \in N_{\Gamma}(H)$, $\phi (hg)=$$\phi (gh)=$$\phi (g)$ holds. Here $N_{\Gamma}(H)$ means the normalizer of $H$ in $\Gamma$. 
\end{lem}

\begin{proof}
We will only show $\phi (hg)=\phi (g)$. Let $\Delta$ be the defect of $\phi$. 
By employing the condition $gHg^{-1}<H$ repeatedly, one has that for any $m \in \mathbb{N}$, there exists an element $h' \in H$ such that $(hg)^m =h' g^m $. 
Hence, one obtains 
$
m \cdot |  \phi (hg) - \phi (g) | \leq \Delta.
$ This ensures the conclusion. 
\end{proof}

\section{\textbf{Proof of Theorem~\ref{thm:vanish}}}\label{sec:vanish}
Let $\phi$ be any homogeneous quasi-homomorphism on $\Gamma$. Because the subgroup $G=EL_2 (A)$ is generated by 
elementary matrices of degree $2$ and $\phi$ vanishes on such matrices (Lemma~\ref{lem:triang}), it suffices to show the following claim for the proof of Theorem~\ref{thm:vanish}:

\bigskip
\textbf{Claim.}  \textit{For any }$g\in G$ \textit{and any} $s \in U_2 A(\subset G)$\textit{,} 
\textit{the equality} $\phi (sg)=\phi (g)$ \textit{holds}. \textit{The same thing holds for }$s\in L_2 A (\subset G)$. 
\bigskip

To prove Claim, first we utilize a result of R. K. Dennis and L. N. Vaserstein \cite[Lemma 18]{DV} (for $k=3$).

\begin{lem}\label{lem:elemgen}$($\cite{DV}$)$
Let $B$ be an associative ring with unit and  $p , q ,r \in GL_1 (B)$ such that $pqr=1$. Let $D$ be the diagonal matrix in $GL_3 (B)$ with the diagonal part $p , q $, and $r$. Then $D \in (L_3 B)( U_3 B)( L_3 B)( U_3 B)$. Here $GL_n$ means the group of all invertible matrices in $M_n$. 
\end{lem}

Let $g$ and $s$ be as in Claim. We set $B= M_2 (A)$ in the lemma above and set 
$p= sg $,\, $q = g^{-1}$ and $r =s^{-1}$. We need the following explicit form: 

\begin{align*} 
 \left(
\begin{array}{ccc}
I_2 & 0 & 0 \\
p^{-1}  & I_2 &  0 \\
0  & q^{-1}  & I_2 
\end{array}
\right) &
\left(
\begin{array}{ccc}
p & 0 & 0 \\
0  & q &  0 \\
0  & 0  & r 
\end{array}
\right) 
= X_1 Y X_2, \ \mathrm{where\ } 
X_1= \left(
\begin{array}{ccc}
I_2 & I_2-p & 0 \\
0  & I_2 &  I_2-pq \\
0  & 0  & I_2 
\end{array}
\right)^{-1}, \\
 Y & =\left(
\begin{array}{ccc}
I_2 & 0 & 0 \\
I_2  & I_2 &  0 \\
0  & I_2  & I_2 
\end{array}
\right) , \ \mathrm{and\ }
X_2=\left(
\begin{array}{ccc}
I_2 & (I_2 -p)q & 0 \\
0  & I_2 &  (I_2 -pq)r \\
0  & 0  & I_2 
\end{array}
\right) .
\end{align*}

By applying Lemma~\ref{lem:hom} and Lemma~\ref{lem:gene} (for $H=L_3 B$), we conclude that the  evaluation of the left-hand side of the first equality in above by $\phi$ is equal to 
$$
\phi (p) + \phi (q) + \phi (r) = \phi (sg) - \phi (g) -\phi(s) = \phi (sg) - \phi (g).
$$
(Here we see $p, q, r, sg,g,$ and $s$ as elements in $G$. It has no problem because each of three diagonal $EL_2 (A)$ parts in $\Gamma$ can be conjugated to each other by permutation matrices.) 
Hence, in order to show Claim, it is enough to show that $\phi (X_1 Y X_2) =0$ for any $g$ and $s$ as in Claim.

\begin{proof}(\textit{Claim}\,)
We may assume that $s\in U_2A$ without loss of generality. Note that $X_1 Y X_2$ is conjugate to $X_2 X_1 Y$. Computation shows that $X_2 X_1 Y= ZXY$, where 
$$
X= \left(
\begin{array}{ccc}
I_2 & -(p -I_2)(q-I_2) & 0 \\
0  & I_2 &  0 \\
0  & 0  & I_2 
\end{array}
\right)  \mathrm{and\ }
 Z  =\left(
\begin{array}{ccc}
I_2 & 0 & -(p- I_2)(q- I_2)(pq-I_2) \\
0  & I_2 &  0 \\
0  & 0  & I_2 
\end{array}
\right) .
$$
Indeed, the key observation here is that if one regards $X_2 X_1$ as an element in $GL_3 (M_2(A))$, then the $(2,3)$-th entry of $X_2 X_1$ is $0(=0_2)$. (This is because $s+s^{-1} = 2I_2$.) We set $x=$$ -(p -I_2)(q-I_2)=$$-(sg-I_2)(g^{-1}-I_2)  $ and $z=$$-(p- I_2)(q- I_2)(pq-I_2)$$= -(sg-I_2)(g^{-1}-I_2)(s- I_2)$ as elements in $M_2 (A)$. By definition, $g$ and $s$ can be written as 
$$
g= \left(
\begin{array}{cc}
a & b  \\
c & d
\end{array}
\right)
\ \ (a,b,c,d \in A)\ \  \mathrm{and} \ \   
s= \left(
\begin{array}{cc}
1 & f  \\
0 & 1
\end{array}
\right)
\ \ (f \in A).
$$
By substituting these matrix forms of $g$ and $s$ for the expressions of $x$ and $z$, we continue calculations as follows: 
\begin{align*}
x&= -\left(
\begin{array}{cc}
a+fc -1 & b+fd \\
c  & d-1  
\end{array}
\right)
\left(
\begin{array}{cc}
d -1 & -b \\
-c  & a-1  
\end{array}
\right)
= 
\left(
\begin{array}{cc}
* & * \\
0  & *  
\end{array}
\right), \\
z&= 
\left(
\begin{array}{cc}
* & * \\
0  & *  
\end{array}
\right)
\left(
\begin{array}{cc}
0 & f \\
0  & 0  
\end{array}
\right)
=
\left(
\begin{array}{cc}
0 & * \\
0  & 0  
\end{array}
\right)
. 
\end{align*}
Here each $*$ respectively represents  a certain element in $A$.

Next we define a (non-unital) subring $N$ of $M_2 (A)$ by 
$$
N= \left\{ \left(
\begin{array}{cc}
0 & l \\
0  & 0  
\end{array}
\right) : 
l\in A \right\}. 
$$
Obviously $z \in N$ and the following  holds.
\begin{lem}\label{lem:Nilp}
In the setting above, the following hold: 
\begin{itemize}
 \item[$(1)$]  For any $u,v \in N$, $uv =0$.
 \item[$(2)$]  For any $u \in N$, $xu \in N$ and $ux \in N$. 
\end{itemize}
\end{lem}
We also define the following subset in $M_6 (A)$: 
$$
\Gamma_N = \left\{ \left(
\begin{array}{ccc}
I_2 +* & * & *  \\
*  & I_2 +* & *  \\
*  &  *  &  I_2+*
\end{array}
\right) : \ \mathrm{each \ }*\ \mathrm{is\ in\ }N
 \right\}. 
$$
The lemma below is the key to proving Claim. 
\begin{lem}\label{lem:grpnilp}
In the setting above, the following hold:
\begin{itemize}
 \item[$(1)$]  The set $\Gamma_N$ is a subgroup of $\Gamma$. 
 \item[$(2)$]  The elements $X$ and $Y$ 
  normalize $\Gamma_N$. 
 \item[$(3)$]  Any homogeneous quasi-homomorphism $\phi$ on $\Gamma$ is bounded on $\Gamma_N$. This means that $\phi$ vanishes on $\Gamma_N$.
\end{itemize}
\end{lem}
\begin{proof}(\textit{Lemma~\ref{lem:grpnilp}}\,)
$(1)$ and $(2)$ are straightforward from Lemma~\ref{lem:Nilp}. For $(3)$, we observe that any element $\gamma \in \Gamma_N$ can be decomposed as 
$$
\gamma = 
\left(
\begin{array}{ccc}
I_2+*  & * & *  \\
0  & I_2+*  & *  \\
0  &  0  &  I_2+* 
\end{array}
\right)
\left(
\begin{array}{ccc}
I_2  & 0 & 0  \\
*  & I_2  & 0  \\
*  &  *  &  I_2 
\end{array}
\right).
$$
Here each $*$ is in $N$. Lemma~\ref{lem:triang} ends our proof.
\end{proof}
We also need the following simple fact. Let 
$$
T= \left(
\begin{array}{ccc}
I_2  & 0 & 0  \\
0  & -I_2  & 0  \\
0  &  I_2  &  I_2 
\end{array}
\right) \in \Gamma.
$$
\begin{lem}\label{lem:inverse}
In the setting above, $T XT^{-1} = X^{-1}$ and $TYT^{-1} =Y^{-1}$ hold. 
\end{lem}

Now we have all ingredients to complete our proof of Claim. The essential point is that Lemma~\ref{lem:gene} applies to the case that $H=\Gamma_N$, $h=Z$, and $g=XY$. (This follows from Lemma~\ref{lem:grpnilp}.) Thus, we show that 
$$
\phi (ZXY) = \phi (XY).
$$
From Lemma~\ref{lem:inverse}, we also see that $XY$ can be transfromed to $X^{-1}Y^{-1}$ by the conjugation by $T$. Because $X^{-1}Y^{-1}$ is conjugate to $Y^{-1}X^{-1}$, $XY$ is conjugate to its inverse. Finally, from Lemma~\ref{lem:vanish}, we have 
$$
\phi (ZXY) = \phi (XY) =0.
$$
\end{proof}
Therefore, we accomplish the proof of Theorem~\ref{thm:vanish}.  
\begin{flushright}
$\square$
\end{flushright}

\section{\textbf{Proof of Theorem~\ref{thm:ge2} and Theorem~\ref{thm:euclid}}}\label{sec:Thm}
One can immediately deduce Theorem~\ref{thm:ge2} from Theorem~\ref{thm:vanish} in the following way: Thanks to Proposition~\ref{prop:hom}, for the proof of Theorem~\ref{thm:ge2}, it suffices to show that every homogeneous quasi-homomorphism on $\Gamma$ is bounded. One obtains this boundedness from Theorem~\ref{thm:vanish} and bounded generation because ranges of every homogeneous quasi-homomorphism are bounded on the two subsets which appear in the bounded generation in condition $(2)$. This ends the proof  of Theorem~\ref{thm:ge2}.

To complete our proof of Theorem~\ref{thm:euclid}, it only remains to prove that a euclidean ring $R$ enjoys the two conditions of Theorem~\ref{thm:ge2} (for all $n$).
\begin{proof}(\textit{Theorem~\ref{thm:euclid}}\,)
It is obvious that $R$ satisfies condition $(1)$. For condition $(2)$, we will use a result by M. Newman, which can be found in the proof of \cite[Theorem2]{New}. 
 
\begin{thm}\label{thm:Newman}$($\cite{New}$)$
Let $R$ be a $($commutative$)$ principal ideal ring and $m$, $l$ be positive integers with $m \geq 3l$. Then any element in $\Gamma =SL_m (R)$ can be expressed as a product of two commutators in $\Gamma$ and some element in $G=SL_{m-l}(R)$, where we regard $G$ as the subgroup of $\Gamma$ in the left upper corner. 
\end{thm}

By applying  Theorem~\ref{thm:Newman} repeatedly (first we start from the case that $m=3$ and $l=1$), we obtain that  for $n\geq 3$, $SL_n (R)$ is boundedly generated by $SL_2 (R)$ and the set of all single commutators. This ends our proof.
\end{proof}

\begin{rem}
Let $k$ be a nonnegative integer and $A_k =\mathbb{Z}[x_1 , \ldots , x_k]$. The group $EL_n (A_k)$ is called the \textit{universal lattice} by Y. Shalom \cite{Shal}. It has a significant role because every group of the form $EL_n (A)$ for an associative, commutative and finitely generated ring $A$ with unit can be realized as a quotient group of some universal lattice. The Suslin stability theorem in \cite{Sus} states that if $n \geq 3$, then $EL_n (A_k)$ coincides with $SL_n (A_k)$. In contrast, for $k\geq 1$, $EL_2 (A_k)$ is very small in $SL_2 (A_k)$. (For details and precise meaning, see \cite{GMV} and \cite{Coh}.) A deep theorem of L. Vaserstein \cite{Vas} states that for $n\geq 3$, the universal lattice $\Gamma =SL_n (A_k)$ is boundedly generated by $G=SL_2 (A_k)$ and the set of all elementary matrices in $\Gamma$.

Suppose the following stronger variant of Vaserstein's bounded generation is valid: ``bounded generation of universal lattices by $G'=EL_2 (A_k)$ and  the set of all elementary matrices in $\Gamma$". Then it would imply the following notable ``corollary": ``For $n\geq 6$, the universal lattice $SL_n (A_k)$ does not admit unbounded quasi-homomorphisms." However, the author has no idea whether that bounded generation is true. (Actually, to ascertain the ``corollary", the bounded generation of universal lattices by $G'$ and the set of all single commutators is sufficient. Nevertheless, the author does not know again whether this weaker one holds.)
\end{rem}

\section*{acknowledgments}
The author thanks his supervisor Narutaka Ozawa for advice. He is also grateful to Professor Koji Fujiwara, Professor Nicolas Monod, and Professor Leonid Vaserstein for fruitful comments. The author is supported by Research Fellowships of the Japan Society for the Promotion of Science for Young Scientists No.20-8313.

\bigskip

\begin{scriptsize}
\begin{center}
Graduate School of Mathematical Sciences, 
University of Tokyo, Komaba, Tokyo, 153-8914, Japan 

\noindent
e-mail: mimurac@ms.u-tokyo.ac.jp

\end{center}
\end{scriptsize}

\begin{thebibliography}{99}
\begin{footnotesize}

\bibitem{Bav}Bavard, C. 
``Longueur stable des commutateurs." 
\textit{L'Enseignement Math\'{e}matique} 37, nos. 1-2 (1991): 109--50




\bibitem{BM} Burger, M. and Monod, N. 
``Continuous bounded cohomology and application to rigidity theory." 
\textit{Geometric and Functional Analysis} 12, no.2 (2002): 219--80

\bibitem{Cal} Calegari, D. 
\textit{scl}, Mathematical Society of Japan Memoirs, 20. Tokyo: Mathematical Society of Japan, 2009


\bibitem{CF} Caprace, P.-E.  and Fujiwara, K. 
``Rank one isometries of buildings and quasi-morphisms of Kac--Moody groups." preprint arxiv:0809.0470, 2008


\bibitem{Coh} Cohn, P. M. ``On the structure of the $GL_2$ of a ring." \textit{Publications Math\'{e}matiques of the Institut des Hautes \'{E}tudes Scientifiques} 30 (1966): 5--53


\bibitem{DV}Dennis, R. K.  and Vaserstein, L. N. 
``On a question of M. Newman on the number of commutators." 
\textit{Journal of Algebra} 118 (1988): 150--61



\bibitem{EF} Epstein, D. B. A. and Fujiwara, K. 
``The second bounded cohomology of word-hyperboloc groups." 
\textit{Topology} 36, no. 6 (1997): 1275--89




\bibitem{Gr} Gromov, M. ``Volume and bounded cohomology." 
\textit{Publications Math\'{e}matiques of the Institut des Hautes \'{E}tudes Scientifiques} 56 (1982): 5--99


\bibitem{GMV} Grunewald, F., Mennicke, J. and Vaserstein, L. 
``On the groups $SL_2(\mathbb{Z}[x])$ and $SL_2(k[x,y])$." 
\textit{Israel Journal of Mathematics} 86 (1994): 157--93



\bibitem{Mon1}Monod, N. \textit{Continuous bounded cohomology of locally compact groups.} Lecture notes in Mathematics, Berlin: Springer, 1758, 2001


\bibitem{Mon2}Monod, N. 
``An invitation to bounded cohomology." 
In \textit{International Congress of Mathematicians}, II, 1183--211. Z\"{u}rich : European Mathematical Society, 2006

\bibitem{Mur}Muranov, A. 
``Finitely generated infinite simple groups of infinite commutator width and vanishing stable commutator length." preprint arxiv:0909.2294, 2009


\bibitem{New}Newman, M. 
``Unimodular commutators." 
\textit{Proceedings of American Mathematical Society} 101 (1987): 605--9 



\bibitem{Shal}Shalom, Y. ``Bounded generation and Kazhdan's property (T)." \textit{Publications Math\'{e}matiques of the Institut des Hautes \'{E}tudes Scientifiques} 90 (1999): 145--68



\bibitem{Sus}Suslin, A. A. ``On the structure of the special linear group over polynomial rings." \textit{Mathematics of the USSR-Izvestiya} 11 (1977): 221--38

\bibitem{Kal} van der Kallen, W. 
``$SL_3(\mathbb{C}[x])$ does not have bounded word length." in \textit{Algebraic K-theory, Part I, (Oberwolfach, 1980)}, 357--61, Lecture Notes in Mathematics, 
Vol. 966. Berlin: Springer, 1982

\bibitem{Vas}Vaserstein, L. ``Bounded reduction of invertible matrices over polynomial ring by addition operators." preprint, 2006

\end{footnotesize}
\end{thebibliography}
\end{document}